\documentclass{amsart}
\usepackage{amssymb,amsmath}
\usepackage{graphicx}
\usepackage{subfigure}
\usepackage{upgreek}
\usepackage{tikz}
\usepackage{tikz-cd}
\usepackage{hyperref}
\usepackage{mathtools}
\usetikzlibrary{matrix}
\usepackage{cmap}
\usepackage[T1]{fontenc}

\usepackage[utf8]{inputenc}
\usetikzlibrary{matrix}
\usepackage[all]{xy}
\usepackage{optparams,eurosym}
\usepackage{amsthm}
\usepackage{amsmath,amssymb}
\usepackage{dbnsymb}
\newtheorem{theorem}{Theorem}[section]
\newtheorem{example}[theorem]{Example}
\newtheorem{corollary}[theorem]{Corollary}
\newtheorem{lemma}[theorem]{Lemma}
\newtheorem{proposition}[theorem]{Proposition}
\newtheorem{definition}[theorem]{Definition}
\newtheorem{non-theorem}{Non-Theorem}

\newtheorem*{remark}{Remark}

\newcommand{\qgroup}{U_{q}(\mathfrak{gl}(1|1))}
\newcommand{\q}{\textbf{q}}
\newcommand{\R}{\check{R}}

\numberwithin{equation}{section}

\begin{document}
	\title {Twistings and the Alexander polynomial}
	\author{Daren Chen}
	\maketitle
	\begin{abstract}
		We give an explicit formula of the Alexander polynomial of the link obtained by adding an arbitrary number of full twists to positively oriented parallel $n$-strands in terms of the Alexander polynomials of the links obtained by adding $0,1,...,n-1$ full twists. From this, we see that the Alexander polynomials stabilize after adding sufficiently many full twists. The main tool used in the computation is expressing the Alexander polynomial using the vector space representation of  $\qgroup$.
	\end{abstract}
\section{Introduction}

In this paper we study the change in the Alexander polynomials of links under inserting full twists. More explicitly, suppose $L$ is an oriented link with some specific link diagram, such that some part of the link diagram consists of $n$-parallel strands of the same orientation. Let $\mathcal{L}$ denote  this particular link diagram of $L$ with this choice of $n$-parallel strands. Let $\mathcal{L}_m$ denote the link obtained from $L$ by inserting $m$ full twists along the $n$-parallel strands and let $\mathcal{L}_0=L$. We obtain the following expression of the Alexander polynomial $\Delta(\mathcal{L}_m)$ of $\mathcal{L}_m$ in terms of $\Delta(\mathcal{L}_0),...,\Delta(\mathcal{L}_{n-1})$.
\begin{figure}[h]
	\[
	{
		\fontsize{8pt}{10pt}\selectfont
		\def\svgwidth{2.5in}
		%% Creator: Inkscape 1.0.2 (e86c870, 2021-01-15), www.inkscape.org
%% Accompanies image file 'full_twist.pdf' (pdf, eps, ps)
\begingroup%
  \makeatletter%
  \providecommand\color[2][]{%
    \errmessage{(Inkscape) Color is used for the text in Inkscape, but the package 'color.sty' is not loaded}%
    \renewcommand\color[2][]{}%
  }%
  \providecommand\transparent[1]{%
    \errmessage{(Inkscape) Transparency is used (non-zero) for the text in Inkscape, but the package 'transparent.sty' is not loaded}%
    \renewcommand\transparent[1]{}%
  }%
  \providecommand\rotatebox[2]{#2}%
  \newcommand*\fsize{\dimexpr\f@size pt\relax}%
  \newcommand*\lineheight[1]{\fontsize{\fsize}{#1\fsize}\selectfont}%
  \ifx\svgwidth\undefined%
    \setlength{\unitlength}{197.67734284bp}%
    \ifx\svgscale\undefined%
      \relax%
    \else%
      \setlength{\unitlength}{\unitlength * \real{\svgscale}}%
    \fi%
  \else%
    \setlength{\unitlength}{\svgwidth}%
  \fi%
  \global\let\svgwidth\undefined%
  \global\let\svgscale\undefined%
  \makeatother%
  \begin{picture}(1,0.64037158)%
    \lineheight{1}%
    \setlength\tabcolsep{0pt}%
    \put(0,0){\includegraphics[width=\unitlength,page=1]{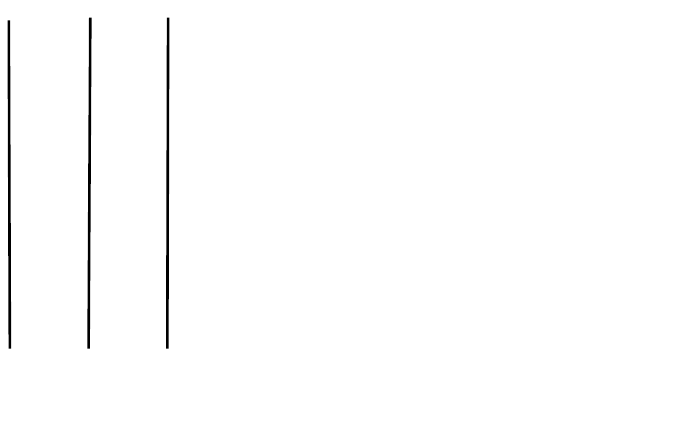}}%
    \put(0.07797883,0.01059839){\makebox(0,0)[lt]{\lineheight{1.25}\smash{\begin{tabular}[t]{l}$L_0$\end{tabular}}}}%
    \put(0,0){\includegraphics[width=\unitlength,page=2]{full_twist.pdf}}%
    \put(0.31526897,0.38710768){\makebox(0,0)[lt]{\lineheight{1.25}\smash{\begin{tabular}[t]{l}Adding one full twist\end{tabular}}}}%
    \put(0,0){\includegraphics[width=\unitlength,page=3]{full_twist.pdf}}%
    \put(0.85691492,0.0094525){\makebox(0,0)[lt]{\lineheight{1.25}\smash{\begin{tabular}[t]{l}$L_1$\end{tabular}}}}%
    \put(0,0){\includegraphics[width=\unitlength,page=4]{full_twist.pdf}}%
  \end{picture}%
\endgroup%

	}
	\]
	\caption{A full twist along $3$-parallel strands}
	\label{fig:cobordism-a}
\end{figure}

\begin{proposition}
	\label{prop:alexpoly}
	 For each $m\geq 0$, there exist some Laurent polynomials $f_{m,j,n}(t)$ which only depend on $m,j,n$ but not on the link $L$,  such that \[\Delta(\mathcal{L}_m) = \sum_{j=0}^{n-1}f_{m,j,n}(t)\Delta(\mathcal{L}_j).\]
\end{proposition}
See Definitions \ref{def:vandermonde} and \ref{def:coefficients} for the explicit definition of $f_{m,j,n}(t)$. By examining the definition of $f_{m,j,n}(t)$, we get the following stabilization result for the Alexander polynomial $\Delta(\mathcal{L}_m)$ as $m\to \infty$.

\newpage
\begin{proposition}
	\label{prop:alexstab}
	 The Alexander polynomials $\Delta(\mathcal{L}_m)$ stabilize as $m\to \infty$ in the following sense:
	 
	 For each link diagram  $\mathcal{L}$ with a chosen part consisting of $n$-parallel strands, there exists a Laurent series $h_{\mathcal{L}}(t)$ with finitely many terms of negative degree in $t$, and some integer $r \in \left[\frac{n-1}{2},n-1\right]$, such that for any $k\in \mathbb{N}$, there exists $N\in \mathbb{N}$ where for any $m\geq N$, the first $k$ terms in the increasing order of degree of $t$ of $\Delta(\mathcal{L}_m)$ agree with the first $k$ terms of  \[t^{mn(n-1-2r)/2}h_{\mathcal{L}}(t).\]
\end{proposition}
By the symmetry of the Alexander polynomials, a similar result holds for the last $k$ terms as well. 

Since the Laurent series $h_{\mathcal{L}}(t)$ is defined as some $\mathbb{C}(q)$-linear combination of the Alexander polynomials $\Delta(\mathcal{L}_m)$ for $m\in \{0,1,2,...,n-1\}$, it satisfies the usual skein relations for Alexander polynomials if we change a crossing away from the chosen $n$-parallel strands. We can think of $h_\mathcal{L}(t)$ as an invariant of the $(n,n)$-tangle obtained by cutting $L$ along the $n$-parallel strands.

The main tool used in this paper is the formulation of the Alexander polynomial as a $\qgroup$-quantum invariant. See \cite{sartori2015alexander} for a detailed explanation of this point of view. In brief, we study the $\qgroup$-equivariant maps induced by oriented tangles on $V^{\otimes n}$, where $V$ is the vector space representation of $\qgroup$. Here we use the assumption that all the $n$-parallel strands are of the same orientation. (In general, we would also need the dual space $V^*$ of $V$.) Also, as all the tangles are oriented upwards in our situation, we can ignore the orientation of the tangle, and consider the action of upward-oriented tangles as an action of the braid group $B_n$ on $V^{\otimes n}$, denoted by \[\Phi: B_n \to End_{\qgroup}(V^{\otimes n}).\] 

 The $\qgroup$-module $V^{\otimes n}$ decomposes as a direct sum
	\begin{equation*}
	V^{\otimes n} \cong \bigoplus^{n-1}_{k=0} {{n-1}\choose{k}}L\big((n-k)\epsilon_1+k\epsilon_2\big),
\end{equation*} 
where each $L\big((n-k)\epsilon_1+k\epsilon_2\big)$ is a $2$-dimensional irreducible representation of $\qgroup$, generated by a highest weight vector with weight $(n-k)\epsilon_1+k\epsilon_2$ as a $\qgroup$-module. Then it is enough to study the action of $B_n$ on the vector space spanned by highest weight vectors of $V^{\otimes n}$ of weights $(n-k)\epsilon_1+k\epsilon_2$ for each $k\in\{0,..,n-1\}$. Denote this space by
	\begin{equation*}
	H_k = \left\{v\in V^{\otimes n} \mid E(v)=0, \,\, \q^h v= q^{\langle h,(n-k)\epsilon_1 + k\epsilon_2\rangle}v\right\}.
\end{equation*}
First, we define a map \[\psi_k:\wedge^k H_1 \to H_k,\] and prove that $\psi_k$ is an isomorphism between $\mathbb{C}(q)$-vector spaces. This is proved by choosing a good basis of $H_k$. See Proposition \ref{prop:basis} and Proposition \ref{prop:wedge product} for the details.

Then we show $\psi_k$ is `almost' an isomorphism of $B_n$-modules, `almost' because they are different by some powers of $q$. More explicitly,
\begin{proposition}
	\label{prop:braid}
	 	For any $k=1,...,n-1$, any $v\in \wedge^k H_1$ and any one-crossing generator $\sigma_t\in B_n$, we have: 
	 \begin{equation*}
	 	\psi_k(\sigma_t\cdot v) = q^{k-1} \sigma_t\cdot \psi_k(v),
	 \end{equation*}
	 where the action of $\sigma_t$ on $\wedge^k H_1$ is the diagonal action of $\sigma_t$ on each component of  $\wedge^k H_1$.
\end{proposition}
As a result, the action of $B_n$ on $V^{\otimes n}$ is completely determined by its action on the vector space $H_1$. It is not hard to compute the action of the full twist $\tau\in B_n$ on $H_1$ explicitly, which is simply a scalar multiplication by $q^{n(n-3)}$. Hence, if we define 
\[\pi_k : V^{\otimes n} \to V^{\otimes n} \,\, \text{for }k\in \{0,...,n-1\},\] as the projection to the subspace $H_k\oplus F(H_k)$, then we can write 
	\begin{equation*}
	\Phi(\tau) = \sum_{k=0}^{n-1}q^{n(n-1-2k)}\pi_k.
\end{equation*}

The equation holds for any power of $\tau$ by raising the coefficients to the corresponding power. By looking at the first $n$ powers of $\tau$ and inverting the coefficient matrix, we can write $\pi_k$ as a $\mathbb{C}(q)$-linear combinations of $\Phi(\tau^0),...,\Phi(\tau^{n-1})$ for each $k\in \{0,...,n-1\}$. Therefore, for any $m\geq 0$, we can write $\Phi(\tau^m)$ as a $\mathbb{C}(q)$-linear combinations of $\Phi(\tau^0),...,\Phi(\tau^{n-1})$. This leads to the expression of $\Delta(L_m)$ in Proposition \ref{prop:alexpoly} in terms of $\Delta(L_0),...,\Delta(L_{n-1})$.

Here are some directions of future work. 
Firstly, the construction in this paper relies on the assumption that the $n$-parallel strands are oriented in the same direction. It is a natural question to ask what happens if we reverse the direction of some of the strands.

Secondly, in \cite{lambert2018twisting}, Lambert-Cole proved stabilization results for knot Floer homology under adding twists to $2$-parallel strands. We would like to explore similar stabilization results for knot Floer homology under adding twists to $n$-parallel strands.  

Thirdly, in \cite{rozansky2010categorification}, Rozansky used the stabilization results on Khovanov homology under adding twists to define Khovanov homology for links in $S^2\times S^1$. See also \cite{willis2021khovanov}, where Willis extended it to links in $\#^rS^2\times S^1$. We would like to see, if such stabilization results hold for knot Floer homology, what is the relation of the stabilization limit with the knot Floer homology for links in $S^2\times S^1$.

\vspace*{2mm}
\textbf{Organization of the paper}. In Section \ref{sec:qgroup} and \ref{sec:rep}, we briefly review the quantum group $\qgroup$ and its representations. In Section \ref{sec:hk}, we give a basis of $H_k$, and prove $H_k\cong \wedge^kH_1$ as vector spaces. In Section \ref{section:braid}, we prove Proposition \ref{prop:braid}. In Section \ref{section:Alex}, we compute the action of the full twist $\tau$ on $V^{\otimes n}$, and then prove Proposition \ref{prop:alexpoly} and \ref{prop:alexstab}.

\vspace*{2mm}
\textbf{Acknowledgements.} The author wants to thank Ciprian Manolescu and Mike Willis for many helpful discussions.

\section{Representations of $\qgroup$}

In this section, we first give a brief review of representations of $\qgroup$. We follow the exposition in Section $2$ and $3$ in \cite{sartori2015alexander} closely. The relatively new material in this section is that we introduce some $\mathbb{C}(q)$-vector space basis of the tensor product $V^{\otimes n}$ of the vector space representation $V$, which will be important for our analysis of the braid group action later.
\subsection{The quantum enveloping subalgebra $\qgroup$}
\label{sec:qgroup}
Let $\mathfrak{gl}(1|1)$ be the Lie superalgebra of linear endomorphisms of the $2$-dimensional graded  complex vector space $\mathbb{C}^{1|1} = \langle u_0,u_1\rangle$ with grading $|u_0|=0, |u_1|=1$. Let $\mathfrak{h} = \langle h_1,h_2\rangle$ be the Cartan subalegbra of the diagonal matrices. Denote the weight lattice by $P=\mathbb{Z}\epsilon_1\oplus \mathbb{Z}\epsilon_2$, where $\left\{\epsilon_1,\epsilon_2\right\}$ is the basis dual to $\left\{h_1,h_2\right\}$. Denote its dual lattice by $P^* = \mathbb{Z}h_1\oplus \mathbb{Z}h_2$.
 The roots of $\mathfrak{gl}(1|1)$ are $\alpha = \epsilon_1-\epsilon_2$ and $-\alpha$.

The quantum enveloping superalgebra $\qgroup$ is the unital superalgebra over the field $\mathbb{C}(q)$ generated by $E,F,\textbf{q}^h$ with $h \in P^*$, with degrees $|\textbf{q}^h|=0$, $|E|=|F|=1$ and relations 
\begin{equation}
	\begin{split}
		\q^0=1,\quad \q^h\q^{h'} &= q^{h+h'}\quad \text{for } h,h'\in P^*,\\
        \q^hE = q^{\langle h,\alpha \rangle}E\q^h, \quad \q^hF & = q^{\langle h,-\alpha \rangle}F\q^h, \quad \text{for }h\in P,\\
        EF+FE = \frac{K-K^{-1}}{q-q^{-1}}, & \quad \text{where }K = \q^{h_1+h_2}, \\
        E^2&=F^2=0.
	\end{split}
\end{equation}

The quantum enveloping superalgebra $\qgroup$ has a Hopf superalgebra structure, with the comultiplication $\Delta:\qgroup \to \qgroup \otimes \qgroup$, the counit $\textbf{u}:\qgroup \to \mathbb{C}(q)$ and the antipodal map $S:\qgroup \to \qgroup$ defined as follows:
\begin{equation}
	\begin{split}
\Delta(E) = E\otimes K^{-1}+1\otimes E,\quad \Delta(F) &= F\otimes 1 + K \otimes F, \quad \Delta(\q^h) = \q^h\otimes \q^h,\\
\textbf{u}(E) = \textbf{u}(F)=0, &\quad\textbf{u}(\q^h)=1,\\
S(E) = -EK, \quad S(F) =&-K^{-1}F, \quad S(\q^h) = \q^{-h}.
	\end{split}
\label{equation:comul}
\end{equation}

\subsection{Representations of $\qgroup$}
\label{sec:rep}

For a weight $\lambda = c_1\epsilon_1+c_2\epsilon_2\in P$, define the grading of $\lambda$ by \[|\lambda|= c_2\mod 2\,\,.\]

 Irreducible representations $L(\lambda)$ of $\qgroup$ are indexed by the their highest weight $\lambda\in P$, which are divided into two cases, depending on whether or not $\lambda $ belongs to $P' := \left\{\lambda\in P | \langle \lambda,h_1+h_2\rangle\neq 0 \right\}$, which is the complement of $\mathbb{Z}\alpha$ in the weight lattice $P$.
\begin{enumerate}
	\item If $\lambda\notin P'$, i.e. $\langle\lambda,h_1+h_2\rangle=0$, then $L(\lambda) $ is a $1$-dimensional $\mathbb{C}(q)$-vector space spanned by $v^{\lambda}$, with grading $|v^{\lambda}| = |\lambda|$ and actions 
	\begin{equation}
		Ev^{\lambda} = Fv^{\lambda}=0,\quad \q^hv^{\lambda} = q^{\langle h,\lambda\rangle }v^{\lambda}.
	\end{equation} 
\item If $\lambda \in P'$, i.e. $\langle\lambda,h_1+h_2\rangle\neq0$, then $L(\lambda)$ is a $2$-dimensional $\mathbb{C}(q)$-vector space spanned by $\left\{v_0^{\lambda},v_1^{\lambda}\right\}$, with gradings $|v^{\lambda}_0| = |\lambda|, |v^{\lambda}_1| = |\lambda|+1$ and actions 
\begin{equation}
	\label{eq:irred rep}
	\begin{split}
	Ev^{\lambda}_0=0, \quad Fv^{\lambda}_0 = \left[\lambda\right]v^{\lambda}_1, \quad &\q^hv^{\lambda}_0=q^{\langle h,\lambda\rangle }v^{\lambda}_0,\\
Ev^{\lambda}_1=v_0^{\lambda}, \quad Fv^{\lambda}_1 = 0, \quad &\q^hv^{\lambda}_1=q^{\langle h,\lambda-\alpha \rangle }v^{\lambda}_1,\\
	\end{split}
\end{equation}
where $\left[\lambda \right] = \frac{q^{\langle h_1+h_2,\lambda\rangle }-q^{-\langle h_1+h_2,\lambda\rangle }}{q-q^{-1}}.$
\end{enumerate}

The most important representation for us is $L(\epsilon_1)$ and we introduce some special notations for it.
\begin{definition}
	\label{def:vector space rep}
	The \textbf{vector space representation} $V$ of $\qgroup$ is  $L(\epsilon_1)$, spanned by $\left\{v_0,v_1\right\}$ with gradings $|v_0|=0,|v_1|=1$, and actions
	\begin{equation}
			\begin{split}
			Ev_0=0, \quad Fv_0 = v_1, \quad &\q^hv_0=q^{\langle h,\epsilon_1\rangle }v_0,\\
			Ev_1=v_0, \quad Fv_1 = 0, \quad &\q^hv_1=q^{\langle h,\epsilon_2 \rangle }v_1,\\
		\end{split}
	\label{eq:vector rep}
	\end{equation}
In particular, for $K=\q^{h_1+h_2}$, we have $Kv_0 = qv_0$ and $Kv_1=qv_1$.

\end{definition}

The following result describes the decomposition tensor product of certain irreducible representations of $\qgroup$.
\begin{lemma}
	\cite[Lemma 3.3]{sartori2015alexander} Suppose $
	\lambda,\mu,\lambda+\mu \in P'$, then we have 
	\begin{equation}
		L(\lambda) \otimes L(\mu) \cong L(\lambda+\mu)\otimes L(\lambda+\mu-\alpha),
	\end{equation}
where $L(\lambda+\mu)$ is spanned by $\left\{v_0^{\lambda}\otimes v_0^{\mu}, F(v_0^{\lambda}\otimes v_0^{\mu})\right\}$, and $L(\lambda+\mu-\alpha)$ is spanned by $\left\{E(v_1^{\lambda}\otimes v_1^{\mu}), v_1^{\lambda}\otimes v_1^{\mu}\right\}$.
\label{lemma:tensor of reps}
\end{lemma}
Note that the action of $E$ on $L(\lambda)\otimes L(\mu)$ is defined via the comultiplication $\Delta$, and there is a sign adjustment due to the grading, e.g.
\begin{align*}
	E(v_1^{\lambda}\otimes v_1^{\mu}) &= \Delta(E)( v_1^{\lambda}\otimes v_1^{\mu}) =(E\otimes K^{-1}+1\otimes E)(v_1^{\lambda}\otimes v_1^{\mu})\\
	          &=(-1)^{|K^{-1}||v_1^{\lambda}|}E(v^{\lambda}_1)\otimes K^{-1}(v^{\mu}_1)+ (-1)^{|E||v^{\lambda}_1|}v^{\lambda}_1\otimes E(v^{\mu}_1)\\
	          &=v^{\lambda}_0\otimes q^{\langle\mu-\alpha ,-h_1-h_2\rangle }v^{\mu}_1+ (-1)^{|\lambda|+1}v^{\lambda}_1\otimes v^{\mu}_0.
\end{align*}

Using Lemma \ref{lemma:tensor of reps} repeatedly, we get the following decomposition of tensor products of the vector space representations $V^{\otimes n}$.

\begin{lemma}
	\cite[Theorem 6.4]{benkart2013planar} The tensor product $V^{\otimes n}$ decomposes as 
	\begin{equation}
		V^{\otimes n} \cong \bigoplus^{n-1}_{k=0} {{n-1}\choose{k}}L\big((n-k)\epsilon_1+k\epsilon_2\big)
	\end{equation}
\label{lemma:tensor of vector rep}
\end{lemma}

\subsection{A Basis of $V^{\otimes n}$}

\label{sec:hk}
In the rest of this section, we will give an explicit choice of highest weight vectors in $V^{\otimes n}$ with weights $(n-k)\epsilon_1+k\epsilon_2$. See Proposition \ref{prop:basis} and \ref{prop:wedge product}. This construction is partly motivated by the `canonical basis' in \cite{zhang2002quantum}, and tensor products of the regular representation of the symmetric group. Similar constructions appeared in \cite{manion2018decategorification} as well. In Section \ref{section:braid}, we will see that this choice of basis gives a natural description of the action of the braid group $B_n$ on $V^{\otimes n}$.

Fix some $n$ for the rest of the section. We introduce some notations first. 
\begin{definition}
	\label{def:hk}
	For $k \in \{ 0,1,...,n-1\}$, let $H_k$ denote the space of highest weight vectors in $V^{\otimes n}$ with weight $(n-k)\epsilon_1 + k\epsilon_2$,  i.e.,
	\begin{equation}
		H_k = \left\{v\in V^{\otimes n} \mid E(v)=0, \,\, \q^h v= q^{\langle h,(n-k)\epsilon_1 + k\epsilon_2\rangle}v\right\}.
	\end{equation}
Denote the space of highest weight vectors of $V^{\otimes n}$ by \[H := \bigoplus_{k=0}^{n-1}H_k.\]
\end{definition}
Note that if $\left\{v_i\right\}_{i\in\mathcal{I}}$ is a basis of highest weight vectors $H$, then $\left\{v_i\right\}_{i\in\mathcal{I}}\cup \left\{F(v_i)\right\}_{i\in\mathcal{I}}$ is a basis of $V^{\otimes n}$, by the description of the irreducible representations of $\qgroup$ in Equation $\ref{eq:irred rep}$ and the direct sum decomposition in Lemma \ref{lemma:tensor of vector rep}. Therefore, we will focus on finding a basis of $H_k$ for each $k \in \{ 0,1,...,n-1\}$.

It follows directly from Lemma \ref{lemma:tensor of vector rep} that dim $H_k= {{n-1}\choose {k}}$. In particular, $H_0$ is $1$-dimensional, which is spanned by $v_0^{\otimes n}$. For other values of $k$, we introduce some other notations before giving a basis.

\begin{definition}
	\label{def:strings}
	For $k\in \left\{ 0,1,...,n-1\right\}$, let $\mathcal{S}_k$ be the set of strings $s=(a_1,b_1,a_2,b_2,...,b_l,a_{l+1})$ which satisfies the following conditions:
	\begin{equation}
		\mathcal{S}_k = \left\{s=(a_1,b_1,a_2,b_2,...,b_l,a_{l+1}) \bigg| \,\,a_i,b_i \in \mathbb{Z}_{\geq0}, b_i\geq 2, \sum_{i=1}^l (b_i-1) =k, \sum_{i=1}^{l+1} a_i + \sum_{j=1}^l b_j =n\right\}
	\end{equation}

We define the function $\phi:\mathcal{S}_k \to V^{\otimes n}$ by the following expression:
\begin{equation}
	\phi((a_1,b_1,a_2,b_2,...,b_l,a_{l+1})) = v_0^{\otimes a_1}\otimes E(v_1^{\otimes b_1})\otimes v_0^{\otimes a_2}\otimes E(v_1^{\otimes b_2}) \otimes ...\otimes E(v_1^{\otimes b_l})\otimes v_0^{\otimes a_{l+1}},
\end{equation}
where if $a_j=0$ for some $j$, then we drop the part $v_0^{\otimes a_j}$ in the expression.
\end{definition}

\begin{remark}
	Note that $v_0 = E(v_1)$, so a string $s\in \mathcal{S}_k$ could be viewed as a recipe to apply $E$'s to different components of $v_1^{\otimes n}$, where $b_i$ means applying $E$ to $v_1^{\otimes b_i}$,  and $a_i$ means applying $a_i$-many $E$'s, each to a single $v_1$ to get $(E(v_1))^{\otimes a_i} = v_0^{\otimes a_i}.$
\end{remark}
\begin{example}
	When $k=0$, the only possible string in $\mathcal{S}_0$ is the 1-digit string $s=(n)$, for which the corresponding vector $\phi(s) = v_0^{\otimes n}$.
	
	When $k=1$, possible strings in $\mathcal{S}_1$ are those such that $l=1$ and $b_1=2$. Thus, we have 
	\begin{equation}
		\mathcal{S}_1 = \left\{s_{i} = (i,2,n-2-i) \mid i = 0,...,n-2\right\}
	\end{equation}
	Denote $ e_i :=\phi(s_{i-1}) = v_0^{\otimes i-1}\otimes E(v_1\otimes v_1)\otimes v_0^{\otimes (n-1-i)}$ for $i=1,...,n-1$. More explicitly, 
	\begin{equation}
		e_i =  v_0^{\otimes i-1}\otimes ( q^{-1}v_0\otimes v_1 -v_1\otimes v_0)\otimes v_0^{\otimes (n-1-i)}.
	\end{equation}
It is easy to see from the explicit formulae that $\left\{e_i\mid i=1,...,n-1\right\}$ is a linearly independent set of size $n-1$. With the help of Lemma \ref{lemma:in Hk}, we get that $e_i\in H_1$, so $\left\{e_i\mid i=1,...,n-1\right\}$ is a basis of $H_1$.
\label{example:basis}
\end{example}
Now we are ready to state the main result of this section.

\begin{proposition}
	\label{prop:basis}
	For each $k\in\{1,...,n-1\}$, $\phi(\mathcal{S}_k)$ gives a basis of $H_k$. 
\end{proposition}

To prove this proposition, we will first prove the statement for $k=1$. For other values of $k$, we will prove the following statement.

\begin{proposition}
	\label{prop:wedge product}
  For each $k\in\{1,...,n-1\}$, there is an isomorphism of $\mathbb{C}(q)$-vector space\[ \psi_k: \wedge^k H_1 \to H_k,\]
such that $\psi_k$ sends a standard basis of $\wedge^kH_1$ to $\phi(\mathcal{S}_k)$:
\[\left\{\psi_k (e_{i_1}\wedge e_{i_2}\wedge ...\wedge e_{i_k}) \mid 1\leq i_1<i_2<...<i_k\leq n-1\right\}= \phi(\mathcal{S}_k),\] 
where $\left\{e_i\right\}_{1\leq i\leq n-1}$ is the basis of $H_1$ defined as in Example \ref{example:basis}.
\end{proposition}

We begin the proof of Proposition \ref{prop:basis} by showing $\phi(\mathcal{S}_k)$ lies in $H_k$.
\begin{lemma}
	\label{lemma:in Hk}
	For each $k\in\{1,...,n-1\}$ and each $s\in \mathcal{S}_k$, we have $\phi(s)\in H_k$.
\end{lemma}
\begin{proof}
	It is enough to check each $\phi(s)$ is a highest weight vector with the correct weight.
	
	First, we give a more explicit formula for $\phi(s)$. By repeated application of $\Delta(E)$ in Equation (\ref{equation:comul}), one gets \begin{equation}
		\Delta^m(E) = \sum_{i=0}^{m}id^{\otimes (m-i)} \otimes E \otimes (K^{-1})^{\otimes i},
	\end{equation}
and \begin{equation}
	E(v_1^{\otimes m})= \Delta^{m-1}(E)(v_1^{\otimes m}) = \sum_{i=0}^{m-1}(-1)^{i}q^{-(m-1-i)}v_1^{\otimes i} \otimes v_0 \otimes v_1^{\otimes( m-1-i)}.
\end{equation}
Therefore, for each $s=(a_1,b_1,a_2,b_2,...,b_l,a_{l+1}) \in \mathcal{S}_k$ as defined in Definition \ref{def:strings}, \[\phi(s)=v_0^{\otimes a_1}\otimes E(v_1^{\otimes b_1})\otimes v_0^{\otimes a_2}\otimes E(v_1^{\otimes b_2}) \otimes ...\otimes E(v_1^{\otimes b_l})\otimes v_0^{\otimes a_{l+1}}\] is a $\mathbb{C}(q)$-linear combination of terms in $V^{\otimes n}$ such that in each term, $k$ components of the tensor product are $v_1$ and $n-k$ components of the tensor product are $v_0$, as we require \[\sum_{i=1}^l (b_i-1) =k\] for $s\in \mathcal{S}_k$. Since $v_0$ has weight $\epsilon_1$ and $v_1$ has weight $\epsilon_2$, the weight of $\phi(s)$ is $(n-k)\epsilon_1+k\epsilon_2$ as required.

Now we show $E(\phi(s))=0$ for any $s\in \mathcal{S}_k$. Choose a string $s = (a_1,b_1,a_2,b_2,...,b_l,a_{l+1}) \in \mathcal{S}_k$. 

For $2\leq j\leq l-1$, let
\begin{equation}
	\label{eq:cj}
c_j = a_1+\sum_{i=1}^{j-1}(b_i+a_{i+1}), \quad d_j = \sum_{i=1}^{j-1}(b_i-1)
\end{equation}
 and let $c_1=a_1$, i.e. $c_j$ is the sum of entries in $s$ before $b_j$.
 \iffalse
  Recall from Equation \ref{eq:vector rep}, we have 
\begin{equation*}
	E(v_0)=0, \quad Kv_0 = qv_0, \quad Kv_1=qv_1. 
\end{equation*}
\fi

Applying $E$ to $\phi(s)$, we get
\begin{equation*}
	\begin{split}
		    E(\phi(s)) &= \Delta^{n-1}(E)(\phi(s))\\
		    &= \sum_{i=0}^{n-1}\big(1^{\otimes (n-1-i)} \otimes E \otimes (K^{-1})^{\otimes i}\big)(v_0^{\otimes a_1}\otimes E(v_1^{\otimes b_1})\otimes v_0^{\otimes a_2}\otimes E(v_1^{\otimes b_2}) \otimes ...\otimes E(v_1^{\otimes b_l})\otimes v_0^{\otimes a_{l+1}})\\
		    &=\sum_{j=1}^{l}\sum_{i=c_j}^{c_j+b_j-1}\big(1^{\otimes i} \otimes E \otimes (K^{-1})^{\otimes (n-1-i)}\big)(v_0^{\otimes a_1}\otimes E(v_1^{\otimes b_1})\otimes v_0^{\otimes a_2}\otimes E(v_1^{\otimes b_2}) \otimes ...\otimes E(v_1^{\otimes b_l})\otimes v_0^{\otimes a_{l+1}})\\
		    &=\sum_{j=1}^{l}(-1)^{d_j}q^{-(n-c_j-b_j)}v_0^{\otimes a_1}\otimes E(v_1^{\otimes b_1})\otimes v_0^{\otimes a_2}\otimes E(v_1^{\otimes b_2}) \otimes ...\otimes E(E(v_1^{\otimes b_j})) \otimes ...\otimes E(v_1^{\otimes b_l})\otimes v_0^{\otimes a_{l+1}}\\
		    &=0,
	\end{split}
\end{equation*}
where the first two equalities are definitions. The third equality follows from $E(v_0)=0$, so we can ignore those terms where $E$ is applied to $v_0$. The fourth equality follows from $K^{-1}v_0 = q^{-1}v_0$ and $K^{-1}v_1=q^{-1}v_1$. The last equality follows from $E^2=0$.
\end{proof}

\begin{proof}[Proof of Proposition \ref{prop:basis} when $k=1$] It follows from the discussion in Example \ref{example:basis} and Lemma \ref{lemma:in Hk}. We see that $\phi(\mathcal{S}_1)$ is contained in $H_1$ by Lemma \ref{lemma:in Hk}, and $\phi(\mathcal{S}_1)$ is a linearly independent set of the right size in Example \ref{example:basis}, hence $\phi(\mathcal{S}_1)$ is a basis of $H_1$.
\end{proof}

Now we turn to the proof of Proposition $\ref{prop:wedge product}$, from which the general case of Proposition \ref{prop:basis} will follow.

\begin{proof}[Proof of Proposition \ref{prop:wedge product}] On the vector space $V$, we define a product $\cdot:V\otimes V\to V$ by:
	\begin{equation}
		v_0\cdot v_0=v_0, \quad v_0\cdot v_1 = v_1\cdot v_0= v_1,\quad v_1\cdot v_1=0.
	\end{equation}
    We extend this to a graded componentwise product on $V^{\otimes n}$ by the obvious formula:
    \begin{equation}
    	\begin{split}
    		m:V^{\otimes n}\otimes V^{\otimes n}&\to V^{\otimes n}\\
    		 (v_{\alpha_1}\otimes v_{\alpha_2}\otimes ...\otimes v_{\alpha_n}) \otimes (v_{\beta_1}\otimes v_{\beta_2}\otimes ...\otimes v_{\beta_n})&\to (-1)^{\clubsuit}v_{\alpha_1}\cdot v_{\beta_1} \otimes v_{\alpha_2}\cdot v_{\beta_2}\otimes ...\otimes v_{\alpha_n}\cdot v_{\beta_n},
    	\end{split}
    \end{equation} 
where \[\clubsuit = \sum_{i=1}^{n}\sum_{j=i+1}^{n}\beta_i\alpha_j.\]

Define $w_i = v_0^{\otimes (i-1)}\otimes v_1\otimes v_0^{\otimes (n-i)}$ for $1\leq i\leq n$. Let $U$ be the $\mathbb{C}(q)$-vector space spanned by $\left\{w_i \mid 1\leq i\leq n\right\}$. Then for each pair $(i,j)$ with $i<j$, we have \[m(w_i\otimes w_j) = v_0^{\otimes (i-1)}\otimes v_1\otimes v_0^{\otimes (j-i-1)}\otimes v_1\otimes v_0^{\otimes (n-j)}= - m(w_j\otimes w_i),\]
and for each $i$, \[m(w_i\otimes w_i)=0.\]
Therefore, $m$ defines an antisymmetric product on $U$, and we can define an isomorphism 
\begin{equation}
	 \psi:\wedge^*U \to V^{\otimes n},
	 \label{eq:def of psi_1}
\end{equation}
 such that for $1\leq i_1<i_2<...<i_k\leq n$, 
\[\psi(w_{i_1}\wedge w_{i_2}\wedge ...\wedge w_{i_k}) = m^{i_k}(w_{i_1}\otimes w_{i_2}\otimes ...\otimes w_{i_k}) =v_{\alpha_1}\otimes v_{\alpha_2}\otimes ...\otimes v_{\alpha_n},\]
where \begin{equation*}
	v_{\alpha_j} = \begin{cases}
		v_1,  \quad &\text{if } j\in \left\{i_1,i_2,...,i_k\right\},  \\
		v_0,    \quad &\text{otherwise.}       \\
	\end{cases}
\end{equation*}

Define \begin{equation}
	\psi_k: \wedge^k H_1 \to V^{\otimes n}
	\label{eq:def of psi_2}
\end{equation} as the restriction of $\psi$ to $\wedge^k H_1$. 

Recall from Example \ref{example:basis}, we have 
\[e_i =  v_0^{\otimes i-1}\otimes ( q^{-1}v_0\otimes v_1 -v_1\otimes v_0)\otimes v_0^{\otimes (n-1-i)} = -w_i+q^{-1}w_{i+1},\]
and we have shown that $\phi(\mathcal{S}_1)=\left\{e_1,e_2,...,e_{n-1}\right\}$ gives a basis of $H_1$.

We are going to prove that
\begin{equation*}
	 \left\{\psi_k(e_{i_1}\wedge e_{i_2}\wedge ...\wedge e_{i_k}) \mid 1\leq i_1<i_2<...<i_k\leq n-1\right\}= \phi(\mathcal{S}_k).\tag{$\dagger$}\label{eq:basis}
\end{equation*}

Suppose we have shown (\ref{eq:basis}), then as $\phi(\mathcal{S}_k)\subset H_k$, we get $\text{Im}(\psi_k)\subset H_k$. Since $\ker(\psi_k)\subset \ker(\psi)=0$, and  $\dim\wedge^k H_1=\dim H_k = {{n-1}\choose {k}} $, it follows that $\psi_k:\wedge^kH_1\to H_k$ is an isomorphism. 

\iffalse As $\phi(\mathcal{S}_k)$ is the image of a basis under an isomorphism, it follows that $\phi(\mathcal{S}_k)$ is a basis of $H_k$.
\fi

It is left to prove Equation (\ref{eq:basis}), which follows from explicit calculations. Note that for a consecutive sequence of indices $i+1<i+2<...<i+m-1$ of length $m-1$, we have \begin{equation}
	\label{eq:wedge product}
	\begin{split}
		&\psi(e_{i+1}\wedge e_{i+2}\wedge...\wedge e_{i+m-1}) \\
		=& \psi \Big((-w_{i+1}+q^{-1}w_{i+2})\wedge (-w_{i+2}+q^{-1}w_{i+3})\wedge ...\wedge (-w_{i+m-1}+q^{-1}w_{i+m})\Big) \\
	    =& \sum_{j=0}^{m-1}(-1)^jq^{-(m-1-j)}\psi \left(w_{i+1}\wedge w_{i+2}\wedge ...\wedge \widehat{w_{i+1+j}}\wedge ...\wedge w_{i+m}\right)\\
	    =&\sum_{j=0}^{m-1}(-1)^jq^{-(m-1-j)}v_0^{\otimes i}\otimes v_1^{\otimes j}\otimes v_0\otimes v_1^{\otimes( m-1-j)}\otimes v_0^{\otimes(n-i-m) }\\
	    =&v_0^{\otimes i}\otimes E(v_1^{\otimes m})\otimes v_0^{\otimes (n-i-m)} = \phi((i,m,n-i-m)),\\
	\end{split}
\end{equation} 
where $\widehat{w_{i+1+j}}$ means this term is missed in the wedge product.

This shows a consecutive sequence of indices $i+1<i+2<...<i+m-1$ of length $m-1$ in the wedge product $e_{i+1}\wedge e_{i+2}\wedge...\wedge e_{i+m-1}$ corresponds to applying  $E$ to $v_1^{\otimes m}$ under the map $\psi$. On the other hand, if the index $i_{j+1}$ is larger than $i_j+1$ in the wedge product $e_{i_1}\wedge e_{i_2}\wedge ...\wedge e_{i_k}$, then they don't interact with each other when we express them in terms of $w_i$, and we can treat them separately when applying $\psi$.

Now given any string $s=(a_1,b_1,...,b_l,a_{l+1})\in \mathcal{S}_k$, we can get a corresponding wedge product $w_s=e_{i_1}\wedge e_{i_2}\wedge ...\wedge e_{i_k}$ by requiring $i_1<i_2<...<i_k$, and  
\begin{equation}
	\label{eq:psi inverse}
	 \{i_1,i_2,...,i_k\} = \bigcup_{j=1}^{l}\left\{c_j+1,c_j+2,...,c_j+b_j-1\right\},
\end{equation} 
where $c_j$ is defined in Equation (\ref{eq:cj}). In words, it says each $b_j$ in $s$ corresponds to a consecutive sequence $c_j+1<c_j+2<...<c_j+b_j-2$ of length $b_j-1$ in the index set of the wedge product $w_s=e_{i_1}\wedge e_{i_2}\wedge ...\wedge e_{i_k}$. 

It is easy to see this gives a bijection between $\mathcal{S}_k$ and the set of elements \[ \left\{e_{i_1}\wedge e_{i_2}\wedge ...\wedge e_{i_k} \mid 1\leq i_1<i_2<...<i_k\leq n-1\right\}.\] For example, one can construct an inverse map, by grouping the indices $ i_1<i_2<...<i_k$ into blocks of consecutive ones. The $j$th block of length $m$ corresponds to $b_j=m+1$, and $a_j$ is given by the difference between $(j-1)$th pair of non-consecutive indices. Applying the calculation in Equation (\ref{eq:wedge product}) to each consecutive sequence of length $b_j-1$ in the index set of $w_s$ shows that \[\phi(s)=\psi(w_s), \quad \text{for any }s\in \mathcal{S}_k.\]
Hence, we have the Equation \ref{eq:basis}, and finish the proof Proposition \ref{prop:wedge product}.
\end{proof}
\begin{proof}[Proof of Proposition \ref{prop:basis}, for general $k>1$.] It follows from Proposition \ref{prop:wedge product}. As $\phi(\mathcal{S}_k)$ is the image of a basis under an isomorphism of vector spaces, $\phi(\mathcal{S}_k)$ is a basis of $H_k$.
\end{proof}

\section{The action of the braid group $B_n$ on $V^{\otimes n}$}
\label{section:braid}

As described in \cite{sartori2015alexander}, for each oriented tangle, we can associate a $\qgroup$-equivariant map using the $R$-matrix which gives the quasitriangular structure of $\qgroup$. See \cite[Section 4]{sartori2015alexander} for more detailed explanation of the construction. 

In this paper, we will consider the $\qgroup$-equivariant maps associated to $(n,n)$-tangles such that each strands are oriented upwards. We will omit the orientation of the tangles, remembering each strand is oriented upwards, and treat them as elements in the braid group $B_n$. In this way, we get an action of $B_n$ on the $\qgroup$-module $V^{\otimes n} = L(\epsilon_1)^{\otimes n}$. Since the action is $\qgroup$-equivariant, it sends highest weight vectors to highest weight vectors and preserves the weight, so we get an action of $B_n$ on the $\mathbb{C}(q)$-vector space $H_k$ for $k\in\{0,1,...,n-1\}$ defined as in Definition \ref{def:hk}. Conversely, the action of $B_n$ on $V^{\otimes n}$ is determined by its action on each $H_k$ by Schur's lemma. We will prove that the action of $B_n$ on $H_k$ is almost the same as the diagonal action of $B_n$ on the wedge product $\wedge^k H_1$, up to some powers of $q$. Hence, the action of $B_n$ on $V^{\otimes n}$ could be described completely by its action on $H_1$. Similar result is hinted in \cite[Proposition 2.5.1]{manion2018decategorification}. We will use this to compute the effect of adding full twists on $n$-positively oriented parallel strands on Alexander polynomial in Section \ref{section:Alex}.

We begin by recalling the definition of $\qgroup$-equivariant maps associated to oriented tangles. Every strand in the tangle is supposed to be oriented upwards, so we drop the arrows in the tangle diagrams in this section.
 
For a positive crossing $\slashoverback$, we associate the map $\check{R}:V^{\otimes2}\to V^{\otimes2}$, 
\begin{equation}
	\begin{split}
		 	&\check{R}(v_0\otimes v_0) = q v_0\otimes v_0, \quad  \check{R}(v_1\otimes v_0) = v_0\otimes v_1,\\
		 	&\R(v_0\otimes v_1) = v_1\otimes v_0 +(q-q^{-1})(v_0\otimes v_1), \quad \R(v_1\otimes v_1) = -q^{-1}v_1\otimes v_1,\\
	\end{split} \label{eq:Rmatrix}
\end{equation}
where $V= L(\epsilon_1)$ is the vector space representation of $\qgroup$ as in Definition \ref{def:vector space rep}.

Fix some positive integer $n$ throughout the section. For each generator $\sigma_t$ of the braid group $B_n$, we associate the map $\R_{t,t+1}: V^{\otimes n}\to V^{\otimes n} $ by applying $\R$ to the $t$th and $(t+1)$th component of $V^{\otimes n}$, \textit{i.e.},
\begin{equation}
	 \R_{i,i+1} := id^{\otimes (t-1)}\otimes \R \otimes id^{\otimes (n-t-1)}. 
\end{equation}
It follows from the construction that this defines a $\qgroup$-equivariant action of $B_n$ on $V^{\otimes n}$.
We denote the action by 
\[\Phi: B_n\to End_{\qgroup}(V^{\otimes n}).\]

 As explained above, this gives an action of $B_n$ on each $H_k$. See Definition \ref{def:hk} for the definition of $H_k$. We denote this action by \begin{equation*}
	\begin{split}
		 \cdot: B_n\times H_k &\to H_k\\
		 (\sigma,v) &\to \sigma \cdot v.\\
	\end{split}
\end{equation*}

It is easy to describe the action of $B_n$ on $H_1$ with respect to the basis $\left\{e_i\right\}_{1\leq i \leq n-1}$ as introduced in Example \ref{example:basis} via explicit calculations. We record it here for later convenience.

\begin{lemma}
	\label{lemma:act on H1}
	For each $t\in \left\{1,...,n-1\right\}$, we have 

	\begin{equation}
	\sigma_t\cdot e_j = \begin{cases}
		qe_j + e_{j+1},\quad \quad&\text{if  }j = t-1,\\
		-q^{-1}e_j,\quad&\text{if }j=t,  \\
		e_{j-1}+qe_j,\quad &\text{if } j= t+1, \\  qe_j\quad &\text{otherwise, i.e. if }$j$ \text{ is not adjacent to } t.\\
	\end{cases}
\end{equation}
\end{lemma} 

\begin{proof}
	 It follows from direct calculation. For example, if $j = t-1$, then \begin{equation*}
	 	\begin{split}
	 		 \sigma_t\cdot e_{t-1} &= id^{\otimes (t-1)}\otimes \R \otimes id^{\otimes (n-t-1)} (v_0^{\otimes t-2}\otimes ( q^{-1}v_0\otimes \underline{v_1\otimes v_0} -v_1\otimes \underline{v_0\otimes v_0})\otimes v_0^{\otimes (n-t-1)} )\\
	 		 & = v_0^{\otimes t-2}\otimes ( q^{-1}v_0\otimes \underline{v_0\otimes v_1} -qv_1\otimes \underline{v_0\otimes v_0})\otimes v_0^{\otimes (n-t-1)}\\
	 		 & = v_0^{\otimes t-2}\otimes (v_0\otimes \underline{v_1\otimes v_0}-qv_1\otimes \underline{v_0\otimes v_0}+ q^{-1}v_0\otimes \underline{v_0\otimes v_1} -v_0\otimes\underline{ v_1\otimes v_0})\otimes v_0^{\otimes (n-t-1)}\\
	 		 & = qe_{t-1}+e_t,
	 	\end{split}
	 \end{equation*}
 where the underline indicates the $t$th and $(t+1)$th components of the tensor product. The rest cases are similar. 
\end{proof}
Recall the isomorphism \[\psi_k: \wedge ^k H_1 \to H_k\] defined in Proposition \ref{prop:wedge product}. See in particular Equation \ref{eq:def of psi_1} and  \ref{eq:def of psi_2}.  The main result of this section is the following proposition.
\begin{proposition}
	\label{prop:wedge product action}
	For any $k=1,...,n-1$, any $v\in \wedge^k H_1$ and each generator $\sigma_t\in B_n$, we have: 
	\begin{equation}
		\psi_k(\sigma_t\cdot v) = q^{k-1} \sigma_t\cdot \psi_k(v),
		\label{eq:action}
	\end{equation}
where the action of $\sigma_t$ on $\wedge^k H_1$ is the diagonal action of $\sigma_t$ on each component of  $\wedge^k H_1$.
\end{proposition}
\begin{proof}
	 By Proposition \ref{prop:basis} and Proposition \ref{prop:wedge product}, it is enough to check the Equation \ref{eq:action} for $v\in \wedge^k H_1$ of the form 
 \[\left\{\psi_k^{-1}(\phi(s))\mid s\in \mathcal{S}_k\right\}.\]
Take an arbitrary $s = (a_1,b_1,a_2,b_2,...,b_l,a_{l+1})\in \mathcal{S}_k$. Recall $\phi(s)$ defined in Definition \ref{def:strings},
\[	\phi(s) = v_0^{\otimes a_1}\otimes E(v_1^{\otimes b_1})\otimes v_0^{\otimes a_2}\otimes E(v_1^{\otimes b_2}) \otimes ...\otimes E(v_1^{\otimes b_l})\otimes v_0^{\otimes a_{l+1}},\]
$e_i$ as defined in Example \ref{example:basis}, $c_j$ as defined in Equation \ref{eq:cj} and let $v =\psi_k^{-1}(\phi(s))= e_{i_1}\wedge ...\wedge e_{i_k}$ as defined in Equation \ref{eq:psi inverse}. 

Take some $\sigma_t\in B_n$. The proof will be divided into five cases depending on the position of $t$ relative to the positions that $E$'s are applied to in $\phi(s)$. The calculation is straightforward, and the hardest part is to keep check of the indices. 

\begin{enumerate}
	\item If $c_j+b_j<t<c_{j+1}$ for some $j$,\textit{ i.e.}, $\R$ is applied to some components $v_0\otimes v_0$ in $\phi(s)$. 
	
	Using Equation \ref{eq:Rmatrix}, we get \[\sigma_t\cdot \phi(s) = q\phi(s) = q \psi_k(v). \] 
	On the other hand, none of the indices $i_m$ in $v= e_{i_1}\wedge ...\wedge e_{i_k}$ equals $t-1$, $t$ or $t+1$, by the condition $c_j+b_j<t<c_{j+1}$ and Equation \ref{eq:psi inverse}. Then apply Lemma \ref{lemma:act on H1}, we get \[\psi_k(\sigma_t\cdot v) =\psi_k( \sigma_t\cdot e_{i_1}\wedge ...\wedge \sigma_t\cdot e_{i_k}) = \psi_k(qe_{i_1}\wedge ...\wedge qe_{i_k})=q^{k}\psi_k(v) = q^{k-1}\sigma_t\cdot\psi_k(v),\]
    so the proposition holds in this case.
	\item If $c_j<t<c_j+b_j$ for some $j$, \textit{i.e.}, $\R$ is applied to the $(t-c_j)$th and $(t-c_j+1)$th components in $E(v_1^{b_j})$. 
	
	Let $m = t-c_j$. The effect of applying $\sigma_t$ to $\phi(s)$ is the same as applying $\sigma_m$ to the $E(v_1^{b_j})$-tensor component. Note that\begin{equation*}
		\begin{split}
			 \sigma_m \cdot E(v_1^{\otimes b_j})&=\sum_{ r=0}^{b_j-1}(-1)^{r}q^{-(b_j-1-r)}\sigma_m\cdot (v_1^{\otimes r}\otimes v_0\otimes v_1^{\otimes (b_j-1-r)}) \\
			 &= (-1)^{m-1}q^{-(b_j-m)}\sigma_m\cdot (v_1^{\otimes (m-1)}\otimes (v_0\otimes v_1-qv_1\otimes v_0)\otimes v_1^{\otimes b_j-m}) \\
			 &\qquad\qquad \qquad +\sum_{\substack{0\leq r\leq b_j-1\\ r\neq m-1,m}}(-1)^{r}q^{-(b_j-1-r)}\sigma_m\cdot (v_1^{\otimes r}\otimes v_0\otimes v_1^{\otimes (b_j-1-r)})\\
			 & = (-1)^{m-1}q^{-(b_j-m)} v_1^{\otimes (m-1)}\otimes (v_1\otimes v_0 +(q-q^{-1})v_0\otimes v_1-qv_0\otimes v_1)\otimes v_1^{\otimes b_j-m} \\
			 &\qquad\qquad \qquad -q^{-1}\sum_{\substack{0\leq r\leq b_j-1\\ r\neq m-1,m}}(-1)^{r}q^{-(b_j-1-r)}v_1^{\otimes r}\otimes v_0\otimes v_1^{\otimes (b_j-1-r)}\\ 
			 & = -q^{-1}E(v_1^{\otimes b_j}).
		\end{split}
	\end{equation*}
Therefore \[\sigma_t\cdot \phi(s) = -q^{-1}\phi(s)=-q^{-1}\psi_k(v).\]

On the other hand, there exists some index $i_p$ in $v = e_{i_1}\wedge ...\wedge e_{i_k}$ which equals $t$, by Equation \ref{eq:psi inverse}. The action of $\sigma_t$ sends $e_{i_p}$ to $-q^{-1}e_{i_p}$. Now for indices except $i_{p-1},i_p, i_{p+1}$, they can't be equal to $t-1,t$ or $t+1$, and $\sigma_t$ acts on them by multiplication by $q$. For $e_{i_{p-1}}$, we have \begin{equation*}
	\sigma_t\cdot e_{i_{p-1}}=\begin{cases}
		qe_{i_{p-1}} + e_{i_{p}},&\quad \quad \text{if }i_p-i_{p-1} = 1,\\
		qe_{i_{p-1}},&\quad\quad \text{otherwise.}
	\end{cases}
\end{equation*}
For $e_{i_{p+1}}$, we have \begin{equation*}
	\sigma_t\cdot e_{i_{p+1}}=\begin{cases}
		qe_{i_{p+1}} + e_{i_{p}},&\quad \quad \text{if }i_{p+1}-i_p = 1,\\
		qe_{i_{p+1}},&\quad\quad \text{otherwise.}
	\end{cases}
\end{equation*}

From this, we get \[\sigma_t \cdot e_{i_{p-1}}\wedge \sigma_t \cdot e_{i_p}\wedge  \sigma_t \cdot e_{i_{p+1}}=-qe_{i_{p-1}}\wedge q^{-1}e_{i_p}\wedge qe_{i_{p+1}}.\] 
 Therefore, \[\psi_k(\sigma_t\cdot v)=\psi_k(\sigma_t \cdot e_{i_1}\wedge ...\wedge \sigma_t\cdot e_{i_k}) =-q^{k-2} \psi_k(v) = q^{k-1}\sigma_t\cdot\psi_k(v),\]
and the proposition holds in this case.
\item If $t = c_j$ for some $j$, and $a_j\neq 0$, i.e., $\R$ is applied to the first two components of $v_0\otimes E(v_1^{\otimes b_j})$.  
\label{case:3}
Note that
\begin{equation*}
	\begin{split}
		\R_{1,2} ( v_0\otimes E(v_1^{\otimes b_j})) &=  \sum_{r=0}^{b_j-1}(-1)^{r}q^{-(b_j-1-r)}\R_{1,2} (v_0\otimes v_1^{\otimes r}\otimes v_0\otimes v_1^{\otimes b_j-1-r})\\
		&= q^{-b_j+2}v_0\otimes v_0\otimes v_1^{\otimes (b_j-1)} +\sum_{r=1}^{b_j-1}(-1)^{r}q^{-(b_j-1-r)}v_1\otimes v_0\otimes v_1^{\otimes (r-1)}\otimes v_0\otimes v_1^{\otimes (b_j-1-r)}\\
		& \quad\quad\quad\quad  +\sum_{r=1}^{b_j-1}(-1)^{r}q^{-(b_j-1-r)}(q-q^{-1})v_0\otimes v_1^{\otimes r}\otimes  v_0\otimes v_1^{\otimes (b_j-1-r)}\\
		& =qv_0\otimes E(v_1^{\otimes b_j})+ (q^{-1}v_0\otimes v_1-v_1\otimes v_0)\otimes \sum_{r=1}^{b_j-1}(-1)^{r-1}q^{-(b_j-1-r)}v_1^{\otimes (r-1)}\otimes v_0\otimes v_1^{\otimes (b_j-1-r)}\\
		& = qv_0\otimes E(v_1^{\otimes b_j})+E(v_1^{\otimes 2})\otimes E(v_1^{\otimes (b_j-1)}).
	\end{split}
\end{equation*}
 We can define a new string $s'$ as follows. 
 \begin{enumerate}
 	\item If $b_j>2$, let $s' = (a'_1,b'_1,...,b'_{l+1},a'_{l+2})$ such that 
 	\begin{equation*}
 		 \begin{split}
 		 &a'_i = a_i, b'_i=b_i,  \text{ for } i<j,\\
 		 &a'_j = a_j-1,\,b'_j=2, \,a'_{j+1}=0, \,b'_{j+1} = b_j-1, \\
 		& a'_i = a_{i-1}, \,b'_i = b_{i-1}, \text{ for } i>j+1.
 		 \end{split}
 	\end{equation*}
 	 In words, we obtain $s'$ from $s$ by inserting a string $(2,0)$ between $a_j$ and $b_j$ in $s$, then changing $a_j$ to $a'_j=a_j-1$ and $b_j$ to $b'_j=b_j-1$.
 	\item If $b_j=2$, let $s' = (a'_1,b'_1,...,b'_l,a'_{l+1})$ such that 
 	\begin{equation*}
 		\begin{split}
 			&a'_i = a_i, b'_i=b_i,  \text{ for } i<j,\\
 			&a'_j= a_j-1, \,b'_j=2=b_j, \,a'_{j+1} = a_{j+1}+1, \,b'_{j+1} = b'_{j+1},\\
 			& a'_i = a_{i}, b'_i = b_{i}, \text{ for } i>j+1.
 		\end{split}
 	\end{equation*}
 In words, we change $a_j$ to $a'_j= a_j-1$, and $a_{j+1}$ to $a'_{j+1} = a_{j+1}+1$.
 \end{enumerate} 
Then it follows from the above calculation that \[\sigma_t\cdot \phi(s) = q\phi(s) +\phi(s').\]
On the other hand, by Equation \ref{eq:psi inverse}, there exists some index $i_p$ in $v = e_{i_1}\wedge ...\wedge e_{i_k}$ which is equal to $t+1$, and $i_{p-1} < t-1$ as $a_j \neq 0$. So all the indices except $i_p$ are not equal to $t-1,t$ or $t+1$. Hence, 
\begin{equation*}
	\begin{split}
		\psi_k(\sigma_t\cdot v) &= \psi_k(\sigma_t \cdot e_{i_1}\wedge ...\wedge \sigma_t\cdot e_{i_k}) \\
		&= \psi_k(qe_{i_1}\wedge ...\wedge (e_{i_p-1}+qe_{i_p})\wedge ...\wedge qe_{i_k})\\
		& = q^{k-1}(q\psi_k(v) +\psi_k(v')),
	\end{split}
\end{equation*}
where $v' = e_{i'_1}\wedge ...\wedge e_{i'_k}$ is obtained from $v$ by changing $i'_p = i_p-1$ and keeping all the other indices the same. It follows from Equation \ref{eq:psi inverse} that $\phi(s') = \psi_k(v')$, hence 
\[\psi_k(\sigma_t\cdot v) = q^{k-1}(q\psi_k(v)+ \psi_k(v')) = q^{k-1}(q\phi(s)+\phi(s'))=q^{k-1}\sigma_t \cdot \psi_k(v),\]
and the proposition holds in this case.
\item If $t = c_j+b_j$ for some $j$, and $a_{j+1}\neq 0$, i.e., $\R$ is applied to the last two components of $E(v_1^{\otimes b_j})\otimes v_0$.

The computation is similar to the one in Case (\ref{case:3}), and we leave it for the reader.
\item If $t = c_{j+1}$ for some $j$, and $a_{j+1}=0$. This is the same as $t = c_{j}+b_{j}$ and $a_{j+1}=0$. In this case, $\R$ is applied to the $b_j$th and $(b_j+1)$th components of $E(v_1^{\otimes b_j})\otimes E(v_1^{\otimes b_{j+1}})$. 

The computation could be carried out in a similar manner as in Case (\ref{case:3}), and one gets 
\begin{equation*}
	\begin{split}
	\sigma_{b_j}\cdot\left( E(v_1^{\otimes b_j})\otimes E(v_1^{\otimes b_{j+1}}) \right) & = qE(v_1^{\otimes b_j})\otimes E(v_1^{\otimes b_{j+1}}) \\
	& \quad\quad +  E(v_1^{\otimes (b_j+1)})\otimes E(v_1^{\otimes (b_{j+1}-1)})+E(v_1^{\otimes (b_j-1)})\otimes E(v_1^{\otimes (b_{j+1}+1)}).\\
	\end{split}
\end{equation*}
So \[ \sigma_t\cdot \phi(s) = q\phi(s)+\phi(s')+\phi(s'').\]
Here $s'$ is obtained from $s$ by the followings rules depending on whether $b_{j+1}=2$ or not. \begin{enumerate}
	\item If $b_{j+1}>2$, then we change $b_j$ to $b'_j = b_j+1$, and $b_{j+1}$ to $b'_{j+1} = b_{j+1}-1$.
	\item If $b_{j+1}=2$, then we change $b_j$ to $b'_{j}=b_j+1$, $a_{j+1}$ to $a'_{j+1} = a_{j+2}+1$, $b_{j+1}$ to $b'_{j+1} = b_{j+2}$, and letting $a'_i = a_{i+1}$, $b'_i = b_{i+1}$ for $i>j+1$. In words, we delete $a_{j+1}, b_{j+1}$ from $s$ and add $1$ to each of $b_j$ and $a_{j+2}$.
\end{enumerate} 

The string $s''$ is defined similarly, by switching the roles of $b_{j}$ and $b_{j+1}$.

On the other hand, there exists indices $i_p$ and $i_{p+1}$ in $v = e_{i_1}\wedge ...\wedge e_{i_k}$, such that \[i_p = t-1,\quad i_{p+1} = t+1,\] by the conditions $t= c_{j+1} = c_j+b_j$ and Equation \ref{eq:psi inverse}. All the other indices are not equal to $t-1,t$ or $t+1$. Hence,
\begin{equation*}
	\begin{split}
		\psi_k(\sigma_t\cdot v) &= \psi_k(\sigma_t \cdot e_{i_1}\wedge ...\wedge \sigma_t\cdot e_{i_k}) \\
		&= \psi_k(qe_{i_1}\wedge ...\wedge (qe_{i_p} + e_{i_p+1})\wedge (e_{i_{p+1}-1}+qe_{i_{p+1}})\wedge ...\wedge qe_{i_k})\\
		& = q^{k-1}(q\psi_k(v) +\psi_k(v')+\psi_k(v'')),
	\end{split}
\end{equation*}
where $v'$ is obtained from $v$ by changing $i_{p+1} =t+1$ to $i_{p+1}-1=t$, and $v''$ is obtained from $v$ by changing $i_p=t-1$ to $i_p+1=t$. Note that by Equation \ref{eq:psi inverse}, \[\psi_k(v')=\phi(s'),\quad \psi_k(v'') = \phi(s'').\]
Therefore the proposition holds in this case.

\end{enumerate}
\end{proof}

\iffalse
\begin{remark}
	This proposition says we can determine the action of $B_n$ on each $H_k$ via its action on $H_1$. Since $H = \oplus H_k$ generates $V^{\otimes n}$ as a $\qgroup$-module, and the action of $B_n$ is $\qgroup$-equivariant, it means we can determine the action of $B_n$ on $V^{\otimes n}$ via its action on $H_1$. This is what we will do to full twists in the next section. Similar thought has appeared in \cite{manion2018decategorification}, where he described the 'localized' action of the simple crossings.
\end{remark}

\fi

\section{The effect of adding twists to Alexander polynomials}
\label{section:Alex}

\iffalse
I need to describe how to get Alexander polynomial. Then we describe the endormorphism induced by adding twists on $V^{\otimes n}$. Notice that we can describe the projection to the subspace spanned by irreducible represenations of the same highest weight by adding $0,1,...,n-1$ twists. Therefore, we get a closed form formula of the Alexander polynomial of adding $N$ full-twists in terms of the Alexander polynomials of adding $0,1,...,n-1$ full-twists. 
\fi

In this section, we study the effect of adding full twists along $n$-positively oriented parallel strands to the Alexander polynomial. By the analysis in the previous section, especially Proposition \ref{prop:wedge product action}, we can compute the action of $B_n$ on $V^{\otimes n}$ in terms of the action of $B_n$ on $H_1$. For the full twist $\tau\in B_n$, we will see the action of it on $H_1$ is a scalar multiplication by some power of $q$. Hence, $\tau$ acts on $V^{\otimes n}$ by scalar multiplications of different powers of $q$ on different highest weight subrepresentations of $V^{\otimes n}$. See Proposition \ref{prop:action on weight space} for the explicit statement. 

Let $\pi_k$ denote the projection of $V^{\otimes n}$ to the highest weight subrepresentation with weight $(n-k)\epsilon_1+k\epsilon_2$, then from the previous analysis we can express each $\pi_k$ as a $\mathbb{C}(q)$-linear combination of $\Phi(\tau^0),\Phi(\tau^1),...,$ $\Phi(\tau^{n-1})$. Therefore, we can express $\Phi(\tau^m)$ for any $m\in \mathbb{N}$ using $\Phi(\tau^0),\Phi(\tau^1),...,\Phi(\tau^{n-1})$. By the description of the Alexander polynomials in terms of $\qgroup$-equivariant maps, we get a similar result for Alexander polynomials. See Proposition \ref{prop:alex}. We explore some implication of this expression on the stabilization behavior of the Alexander polynomials when we insert enough full twists in Proposition \ref{prop:stab}.

We begin with some definitions. Recall the decomposition \[		V^{\otimes n} \cong \bigoplus^{n-1}_{k=0} {{n-1}\choose{k}}L\big((n-k)\epsilon_1+k\epsilon_2\big)\] in Lemma \ref{lemma:tensor of vector rep}, and the subspace $H_k$ of highest weight vectors with weight $(n-k)\epsilon_1+k\epsilon_2$ in Definition \ref{def:hk}
\begin{definition}
	For each $k\in\{0,1,...,n-1\}$, denote the subspace of highest weight subrepresentation with weight $(n-k)\epsilon_1+k\epsilon_2$ by $W_k$, \textit{i.e.},\[W_k= H_k\oplus F(H_k).\] Define $\pi_k:V^{\otimes n}\to V^{\otimes n}$ as the $\qgroup$-equivariant projection from $V^{\otimes n}$ to the subspace $W_k$.
\end{definition}

Note that $W_k \cong L((n-k)\epsilon_1+k\epsilon_2)^{\oplus{{n-1}\choose {k}}}$, as $H_k$ is the space of highest weight vectors of weight \newline $(n-k)\epsilon_1+k\epsilon_2$, and each copy of $L((n-k)\epsilon_1+k\epsilon_2)$ is spanned by $\{v_0^{(n-k)\epsilon_1+k\epsilon_2}, F(v_0^{(n-k)\epsilon_1+k\epsilon_2})\}$. Therefore, \[V^{\otimes n}\cong \bigoplus_{k=0}^{n-1}W_k,\quad  \text{and} \quad \quad id_{V^{\otimes n}}=\sum_{k=0}^{n-1}\pi_k.\]
\begin{definition}
	Define $\lambda = \sigma_{n-1}\sigma_{n-2}...\sigma_1\in B_m$, and let $\tau = \lambda^n$, which is the braid representing the full twist on $n$-strands. Denote the induced endomorphisms of $\lambda, \tau$ on $V^{\otimes n}$ by $\Phi(\lambda),\Phi(\tau)$ respectively.
\end{definition}

Since the action of $B_n$ on $V^{\otimes n}$ is $\qgroup$-equivariant, it restricts to an action on $W_k$ for each $k$. We will prove that the restriction $\Phi(\tau)|_{W_k}$ acts by a scalar multiplication on each $W_k$.

\begin{proposition}
	\label{prop:action on weight space}
	For each $k\in\{0,1,...,n-1\}$, we have \[\Phi(\tau)|_{W_k} = q^{n(n-1-2k)}id_{W_k}.\]
\end{proposition}
\begin{proof}
	Since $H_k = \ker(E)$, $F(H_k) = \ker(F)$ for $E,F:W_k\to W_k$ respectively, and $\Phi(\tau)|_{W_k}$ commutes with $E,F$, the map $\Phi(\tau)|_{W_k}$ could be written as \[\Phi(\tau)|_{W_k} = \Phi(\tau)|_{H_k} \oplus \Phi(\tau)|_{F(H_k)}.\] 	
	It is enough to check that $\Phi(\tau)|_{H_k} = q^{n(n-1-2k)}id_{H_k},$ as $\Phi(\tau)|_{W_k}$ commutes with $F$.

	When $k=0$, $H_0 $ is spanned by $  v_0^{\otimes n}$. By Equation \ref{eq:Rmatrix}, we have \[\sigma_t\cdot v_0^{\otimes n} = q v_0^{\otimes n},\] for $t\in\left\{1,...,n-1 \right\}$. So \[\tau \cdot v_0^{\otimes n} = q^{n(n-1)} v_0^{\otimes n},\]
	and the proposition holds for $k=0$.
	
	When $k=1$, a basis $\left\{e_i\mid i=1,...,n-1\right\}$ of $H_1$ is given by Example \ref{example:basis}. We know the action of $\sigma_t$ on $e_i$ for any pair $(t,i)$, as described in Lemma \ref{lemma:act on H1}. It is then straightforward to compute that \begin{equation}
		\begin{split}
			&\lambda\cdot e_1 = \sum_{j=1}^{n-1}-q^{n-2-j}e_j, \quad \text{and}\\
			&\lambda\cdot e_i = q^{n-2}e_{i-1}, \quad \text{for } i=2,...,n-1. 
		\end{split}
	\end{equation}
Let $A$ be the $(n-1)\times (n-1)$ matrix representing $\Phi(\lambda)|_{H_1}$, then $A$ is diagonalizable over $\mathbb{C}(q)$, with $n-1$ distinct eigenvalues $\left\{e^{2j\pi i/n}q^{n-3} \mid j=1,...,n-1\right\}.$ Hence the matrix $A^n$ which represents $\Phi(\tau)|_{H_k}$ is a diagonal matrix with every diagonal entry equal to $q^{n(n-3)}$. Therefore, \[ \Phi(\tau)|_{H_1} = q^{n(n-3)} id_{H_1},\] and the proposition holds for $k=1$.

When $2\leq k \leq n-1$, we use the previous calculation in the case $k=1$ and Proposition \ref{prop:wedge product action}. Proposition \ref{prop:wedge product action} says that for each generator $\sigma_t\in B_n$, and every $v = e_{i_1}\wedge ...\wedge e_{i_k}\in \wedge^kH_1$, we have 	\begin{equation*}
	\psi_k(\sigma_t\cdot v) = q^{k-1} \sigma_t\cdot \psi_k(v),
\end{equation*} so 
\begin{equation*}
	\psi_k(\tau \cdot v) = q^{(k-1)n(n-1)}\tau \cdot \psi_k(v),
\end{equation*}
as $\tau$ is the product of $n(n-1)$-many elements of the form $\sigma_t$ in $B_n$. 

By the calculation in the case when $k=1$, we get that
\begin{equation*}
	\psi_k(\tau \cdot v ) = \psi_k(\tau\cdot e_{i_1}\wedge ...\wedge \tau\cdot e_{i_k}) = q^{kn(n-3)}\psi_k(v),
\end{equation*}
 so 
 \begin{equation}
 	\tau \cdot \psi_k(v) = q^{-(k-1)n(n-1)} \psi_k(\tau \cdot v) = q^{n(n-1-2k)}\psi_k(v),\,\,\, \text{for any } v\in \wedge^k H_1.
 \end{equation}
Since $\psi_k:\wedge^k H_1 \to H_k$ is an isomorphism, we get \[\Phi(\tau)|_{H_k} = q^{n(n-1-2k)}id_{H_k},\]
for $2\leq k \leq n-1$ as required.
\end{proof} 
An immediate corollary of this proposition is 
\begin{corollary}
	We can write $\Phi(\tau):V^{\otimes n}\to V^{\otimes n}$ as a $\mathbb{C}[q,q^{-1}]$-linear combination of $\pi_k$, \textit{i.e.},
	\begin{equation}
		\Phi(\tau) = \sum_{k=0}^{n-1}q^{n(n-1-2k)}\pi_k.
	\end{equation}
\label{coro:twist}
\end{corollary}
Now the task is to represent the projection maps $\{\pi_k \mid k=0,...,n-1\}$ in terms of linear combinations of powers of $\Phi(\tau)$, which could be done easily by solving a linear system of equations. This linear system is guaranteed to be invertible by the non-vanishing of the Vandermonde determinant.

\begin{definition}
	\label{def:vandermonde}
	Foe each $n\geq2$, define $B(n)$ as  the $n\times n$ matrix by
	\begin{equation}
		B(n)_{ij} = q^{(i-1)n(n-1-2(j-1))}.
	\end{equation}
	
	Since $B(n)$ is a Vandermonde matrix, it is invertible. Let $C(n)$ denote the inverse matrix of $B(n)$.
	
\end{definition}
Note that while each entry $B(n)_{i,j}$ is a Laurent polynomial in $q$, each entry $C(n)_{i,j}$ is a rational function in $q$ instead, as we divide out the determinant of $B(n)$ when inverting the matrix.

Let $\overrightarrow{\tau} = (\Phi(\tau^0),\Phi(\tau^1),...,\Phi(\tau^{n-1}))^\intercal$, $\overrightarrow{\pi} = (\pi_0,\pi_1,...,\pi_{n-1})^\intercal$, where $\Phi(\tau^0) = id_{V^{\otimes n}}$.  By Corollary \ref{coro:twist},  \[\overrightarrow{\tau} = B(n)\overrightarrow{\pi},\]
so 
  \[\overrightarrow{\pi} = C(n)\overrightarrow{\tau},\]
  and we can express $\Phi(\tau^m)$ for any $m\geq 0$ as a $\mathbb{C}(q)$-linear combination of $\Phi(\tau^0),...,\Phi(\tau^{n-1})$.

\iffalse
where $B(n)$ is a $n\times n$ Vandermonde matrix with
\begin{equation}
	B(n)_{ij} = q^{(i-1)n(n+1-2j)}.
\end{equation} Then $ det(B(n))\neq 0$, and $B(n)$ is invertible. Let $C(n)$ denote the inverse matrix of $B(n)$, then we can express $\tau^m $ for any $m$ as a linear combination of $\tau^0,\tau^1,...,\tau^{n-1}$ using $C(n)_{ij}$.

\fi
\begin{proposition}
	 For any $m\geq0$, we have \begin{equation}
	 	\Phi(\tau^m) = \sum_{i=0}^{n-1}\sum_{j=0}^{n-1}q^{mn(n-1-2i)}C(n)_{i+1,j+1}\Phi(\tau^{j}),
	 \end{equation}
 \label{prop:twist}
 where $C(n)$ is defined in Definition \ref{def:vandermonde}. In particular, $C(n)$ only depends on $n$. 
\end{proposition} 
\begin{proof}
	It follows directly from Corollary \ref{coro:twist}.
\end{proof}

Let $L$ be an oriented link. Recall the Alexander polynomial $\Delta(L)$ could be obtained as follows. Cutting $L$ along a strand to get a $1$-$1$ tangle. This tangle gives a $\qgroup$-equivariant map on $V$, which is a scalar multiplication of the identity. Denote this scalar by $\hat{Q}(L)$. Then \[\Delta(L)=\hat{Q}(L)\mid_{q=t^{1/2}}.\] See \cite[Section 4]{sartori2015alexander} for details.

Let $L$ be an oriented link with some specific link diagram, such that some part of the link diagram consists of $n$-parallel strands of the same orientation. Let $\mathcal{L}$ denote this particular link diagram with this choice of $n$-parallel strands. For $m\geq 1$, let $\mathcal{L}_m$ denote the oriented link obtained by adding $m$-full twists along the chosen $n$-parallel strands.  Denote $\mathcal{L}_0=L$. From the above calculation on the $\tau^m:V^{\otimes n}\to V^{\otimes n}$, and the expression of the Alexander polynomial as a $\qgroup$-invariant, we get a similar relation on the Alexander polynomials of $\mathcal{L}_m$ with the Alexander polynomials of $\mathcal{L}_0,...,\mathcal{L}_{n-1}$.

\begin{proposition}
	For any $m\geq 0$, we have
		\label{prop:alex} \begin{equation}	
		\label{eq:alex poly}
		 \Delta(\mathcal{L}_m) = \sum_{i=0}^{n-1}\sum_{j=0}^{n-1}t^{mn(n-1-2i)/2}C(n)_{i+1,j+1}|_{q = t^{1/2}}\Delta(\mathcal{L}_j).
	\end{equation}
\end{proposition}
\begin{proof}
	It follows from Proposition \ref{prop:twist} and the expression of Alexander polynomials as a $\qgroup$-quantum invariant.
\end{proof}

\iffalse
\begin{example}
	When $n=2$, \[B(2) = \begin{pmatrix}
	 1&1\\
	 q^2&q^{-2}\\
	\end{pmatrix}, \quad C(2)= \dfrac{1}{q^{-2}-q^2}\begin{pmatrix}
    q^{-2}&-1\\
    -q^2 & 1 \\	
\end{pmatrix}, \]
and 

\end{example}
\fi

This expression doesn't look so useful at first glance. In particular, $C(n)_{ij}$ is not a Laurent polynomial, as it has the $set(B(n))$ as its denominator. However, as all the Alexander polynomials are Laurent polynomials in $t$, the coefficient before each $\Delta(\mathcal{L}_j)$ in the expression $(\ref{eq:alex poly})$ is a Laurent polynomial as well. We give them a name for the convenience of later discussion. 

\begin{definition}
	
	\iffalse
 For any oriented link $L$ which has $n$ parallel strand of the same orientation, let $L_m$ denote the link obtained by adding $m-$full twists along the $n$ parallel strands of the same orientation in $L$. 
 \fi
 \label{def:coefficients}
 
 For each pair $(m,j)$ with $m\geq 0$ and $0\leq j \leq n-1$, define the Laurent polynomial $f_{m,j,n}(t)$ by 
 \[f_{m,j,n}(t)= \sum_{i=0}^{n-1}t^{mn(n-1-2i)/2}C(n)_{i+1,j+1}|_{q = t^{1/2}}.\]
\end{definition}

With $f_{m,j,n}(t)$, we can write the Alexander polynomial $\Delta(\mathcal{L}_m)$ as

\begin{equation}
	 \label{eq:alex poly 2}
	 \Delta(\mathcal{L}_m) = \sum_{j=0}^{n-1}f_{m,j,n}(t)\Delta(\mathcal{L}_j).
\end{equation}
Note that the coefficients $f_{m,j,n}(t)$ are determined by $m,j,n$ and don't depend on the specific link $L$. 

As $m\to \infty$, we have the following stabilization results of the expressions $f_{m,j,n}$.
\newpage

\iffalse
\begin{corollary}
	For $m$ sufficiently large, the Alexander polynomial $\Delta(L_m)$ is of the form
	\[\Delta(L_m) = \sum_{i=0}^{n-1}t^{mn(n-1-2i)/2}f(L)_i(t),\]
	for some Laurent polynomials $f(L)_0,...,f(L)_{n-1}\in \mathbb{C}\left[t,t^{-1}\right]$, with no cancellation, i.e. all the terms $t^{mn(n-1-2i)/2}f(L)_i(t)$ are spread out, with no overlapping over the powers of $t$.
\end{corollary}
\fi

\begin{lemma}
	\label{lem:stab}
	
	The coefficients $f_{m,j,n}(t)$  stabilize in the following sense as $m\to \infty$:
	
	There exists a Laurent series $g_{j,n}(t)$ with finitely many terms of negative degree in $t$, such that for any $k\in \mathbb{N}$, there exists $N\in\mathbb{N}$, where for any $m\geq N$, the first $k$ of $f_{m,j,n}(t)$ terms in the increasing order of powers of $t$  equal the first $k$ terms of \[t^{-mn(n-1)/2}g_{j,n}(t).\]

\end{lemma}
\begin{proof}
	 It follows from the expression of of $f_{m,j,n}(t)$. We formally invert the denominator of $C(n)$ as a Laurent series with finitely many terms of negative degree in $t$. Since $C(n)$ is some fixed matrix which only depends on $n$, for large enough $m$, the term with lowest degree of $t$ in the expression 
	 \[\sum_{i=0}^{n-2}t^{mn(n-1-2i)/2}C(n)_{i+1,j+1}|_{q = t^{1/2}},\]
	  is $t^{mn(3-n)/2+k_1}$ for some constant $k_1$, which is obtained when $i = n-2$, and $k_1$ comes from the contribution of $C(n)$.
	  
	   On the other hand, the lowest power of $t$ in $f_{L,m,j}(t)$ is  $t^{mn(1-n)/2+k_2}$ for some constant $k_2$, where $k_2$ again comes from the contribution of $C(n)$. The difference in the powers is $mn+k_1-k_2$, which goes to infinity as $m\to \infty$. Therefore, for $m$ large enough (since there are only finitely many $j$, we can pick $m$ large enough such that the following holds for every $j$), the first $k$ terms of $f_{m,j,n}(t)$ are the same as the first $k$ terms of \[t^{-mn(n-1)/2}C(n)_{n,j+1}|_{q = t^{1/2}}.\]
	   Then \[g_{j,n}(t):=C(n)_{n,j+1}|_{q = t^{1/2}}\] satisfies the requirement of the lemma. From explicit formula of $C(n)_{n,j+1}$, we see $g_{j,n}(t)$ is a non-zero Laurent series.
\end{proof}

Because of Equation \ref{eq:alex poly 2}, and that all the Alexander polynomials $ \Delta(\mathcal{L}_j)$ for $j\in\{0,1,...,n-1\}$ are Laurent polynomials of finite degree, we have similar stabilization results as in Lemma \ref{lem:stab} for the Alexander polynomial $\Delta(\mathcal{L}_m)$ as $m\to \infty$. The difference is that the limiting Laurent series obtained in this way might vanish, in which case we switch to the next leading Laurent series, until we reach a non-vanishing one. This explains the appearance of $r$ in the next proposition.

\begin{proposition}
	\label{prop:stab}
	 The Alexander polynomials $\Delta(\mathcal{L}_m)$ stabilizes as $m\to \infty$ in the following sense:

	 There exists a Laurent series $h_{\mathcal{L}}(t)$ with finitely many terms of negative degree in $t$, and some integer $r \in \left[\frac{n-1}{2},n-1\right]$, such that for any $k\in \mathbb{N}$, there exists $N\in \mathbb{N}$ where for any $m\geq N$, the first $k$ terms in the increasing order of degree of $t$ of $\Delta(\mathcal{L}_m)$ agree with the first $k$ terms of  \[t^{mn(n-1-2r)/2}h_{\mathcal{L}}(t).\]

\end{proposition}
\begin{proof}
	It follows from Proposition  \ref{prop:alex} and Lemma \ref{lem:stab}. The Laurent series $h_{\mathcal{L}}(t)$ is given by
	\begin{equation}
 \sum_{j=0}^{n-1}C(n)_{n,j+1}|_{q =t ^{1/2}}\Delta(\mathcal{L}_j),
		\label{eq:laurent}
	\end{equation} 
and $r = n-1$, given that this Laurent series does not vanish. If it vanishes, then by a similar argument as in Lemma \ref{lem:stab}, the first $k$ terms of $\Delta(\mathcal{L}_m)$ are the same as the first $k$ terms of 
\begin{equation*}
 t^{mn(-n+3)/2}\sum_{j=0}^{n-1}C(n)_{n-1,j+1}|_{q =t ^{1/2}}\Delta(\mathcal{L}_j):=t^{mn(-n+3)/2}h_{n-2,\mathcal{L}}(t),
\end{equation*} 
 given that $h_{n-2,\mathcal{L}}(t)$ doesn't vanish. 

In general, let \[h_{r,\mathcal{L}}(t) :=\sum_{j=0}^{n-1}C(n)_{r+1,j+1}|_{q =t ^{1/2}}\Delta(\mathcal{L}_j). \] If $h_{r,\mathcal{L}}(t)$ doesn't vanish, while $h_{r+1,\mathcal{L}}(t)$, $h_{r+2,\mathcal{L}}(t),...,h_{n-1,\mathcal{L}}(t)$ all vanish, then the first $k$ terms of $\Delta(\mathcal{L}_m)$ are the same as the first $k$ terms of \[t^{mn(n-1-2r)/2}h_{r,\mathcal{L}}(t).\] This is the integer $r$ as stated in the proposition, and $h_{\mathcal{L}}(t) = h_{r,\mathcal{L}}(t).$ If $h_{r},\mathcal{L}(t)$ vanishes for all integers $r \in [\frac{n-1}{2},n-1]$ (note by the symmetry of the Alexander polynomials, it is enough for to assume this holds for $r\geq \frac{n-1}{2}$), then $\Delta(\mathcal{L}_m)$ also vanishes for each $m$, and the statement trivially holds.
\end{proof}

\begin{remark}
	By the symmetry of the Alexander polynomial, we get similar stabilization results for the last few terms of $\Delta(\mathcal{L}_m)$ as well. 
\end{remark}

\begin{remark}
Since the Laurent series $h_{\mathcal{L}}(t)$ is defined as some $\mathbb{C}(q)$-linear combinations of the Alexander polynomials $\Delta(\mathcal{L}(m))$ for $m\in \{0,1,2,...,n-1\}$, it satisfies the usual skein relations for Alexander polynomials if we change a crossing away from the $n$-parallel strands.
\end{remark}

\begin{example}
	When $n=2$, \[C(2)|_{q = t^{1/2}} = \dfrac{1}{t^{-1}-t}\begin{pmatrix}
		t^{-1}&-1 \\
		-t&1	
	\end{pmatrix},\] 
and the Equation \ref{eq:alex poly} becomes \[\Delta(\mathcal{L}_{m}) = \dfrac{-t^{-(m-1)}+t^{m-1}}{ t^{-1}-t}\Delta(\mathcal{L}_0) +\dfrac{t^{-m}-t^{m}}{ t^{-1}-t} \Delta(\mathcal{L}_1),\]
In this case, it is also easy to get the expression of $\Delta(\mathcal{L}_m)$ from the inductive expression \[\Delta(\mathcal{L}_{m+1}) = (t+\dfrac{1}{t})\Delta (\mathcal{L}_m) - \Delta(\mathcal{L}_{m-1}),\] by using the usual skein relation of the Alexander polynomials three times.

The Laurent series $h_{\mathcal{L}}(t)$ is given by 
\[h_{\mathcal{L}}(t) = (t\Delta(\mathcal{L}_1)-t^2\Delta(\mathcal{L}_0))\sum_{i=0}^{\infty}t^{2i}.\]
\end{example}
\begin{example}
When $n=3$, 
   \[C(3)|_{q = t^{1/2}} = \dfrac{1}{t^{-6}-2t^{-3}+2t^{3}-t^6} \begin{pmatrix}
   t^{-6}-t^{-3}&-t^{-6}+1&t^{-3}-1 \\
   	-t^{-3}+t^{3}&t^{-6}-t^6& -t^{-3}+t^3 \\
   	t^3-t^6&-1+t^6&1-t^3	
   \end{pmatrix},\] 
and 

\begin{equation*}
	\Delta(\mathcal{L}_m) = \dfrac{1}{t^{-6}-2t^{-3}+2t^{3}-t^6}(F_0(t)\Delta(\mathcal{L}_0)+F_1(t)\Delta(\mathcal{L}_1)+F_2(t)\Delta(\mathcal{L}_2)),
\end{equation*}
where \begin{equation*}
	\begin{split}
		&F_0(t) = t^{-3m+3}-t^{-3m+6} -t^{-3}+t^{3}+t^{3m-6}-t^{3m-3},\\
		&F_1(t) =-t^{-3m}+t^{-3m+6}+t^{-6}-t^6-t^{3m-6}+ t^{3m},\\
		&F_2(t) =t^{-3m}-t^{-3m+3}-t^{-3}+t^3+t^{3m-3} -t^{3m}.
	\end{split}
\end{equation*}

The Laurent series $h_{\mathcal{L}}(t)$ is 
 \[h_{\mathcal{L}}(t) = \dfrac{(t^9-t^{12})\Delta(\mathcal{L}_0)+(-t^6+t^{12})\Delta(\mathcal{L}_1)+(t^6-t^9)\Delta(\mathcal{L}_2)}{1-2t^{3}+2t^{9}-t^{12}},\] where we expand the inverse of the denominator as a power series in $t$, given that it doesn't vanish. If it vanishes, then \[\Delta(\mathcal{L}_m) = \dfrac{(-t^{-3}+t^3)\Delta(\mathcal{L}_0) +(t^{-6}-t^6)\Delta(\mathcal{L}_1)+(-t^{-3}+t^3)\Delta(\mathcal{L}_2)}{t^{-6}-2t^{-3}+2t^{3}-t^6},\] independent of $m$.

\end{example}

\begin{example}
Suppose $L= T_{n,l}$ is the torus knot where $gcd(n,l)=1$. Let $\mathcal{L}$ be the link diagram of $L$ which is given by the closure of the braid $(\sigma_{n-1}\sigma_{n-2}...\sigma_1)^l\in B_n$, where all strands are oriented in the same direction. In this case, $\mathcal{L}_{m}$ represents the torus knot $T_{n,l+mn}$. From the formula of the Alexander polynomial of torus knots
\begin{equation*}
	\begin{split}
		 \Delta(T_{n,l+mn}) &= t^{-(l+mn-1)(n-1)/2}\dfrac{(1-t^{(l+mn)n})(1-t)}{ (1-t^{l+mn})(1-t^n) }\\
		 &=t^{(n-1)/2} \dfrac{1-t}{1-t^n}\sum_{i=0}^{n-1}t^{(l+mn)(n-1-2i)/2},
	\end{split}
\end{equation*}
one can see that as $m\to\infty$, the first few terms of $\Delta(T_{n,l+mn})$ in the increasing order of degree of $t$ are obtained when $i =n-1$ in the summation, which are the same as the first few terms of \[t^{-(l-1+mn)(n-1)/2}\dfrac{1-t}{1-t^n}.\] 
Therefore, we get directly the limiting Laurent series $h_{\mathcal{L}}(t)$ is given by \[h_\mathcal{L}(t) = t^{-(l-1)(n-1)/2}\dfrac{1-t}{1-t^n}. \]

Note that this is the same as the one we obtained in the proof of Proposition \ref{prop:stab}. By Equation \ref{eq:laurent}, and $\mathcal{L}_j = T_{n,l+jn}$, we have \begin{equation}
	\begin{split}
		h_{\mathcal{L}}(t) &= \sum_{j=0}^{n-1}C(n)_{n,j+1}|_{q=t^{1/2}}\Delta(T_{n,l+jn})\\
		&= \sum_{j=0}^{n-1}C(n)_{n,j+1}|_{q=t^{1/2}}t^{(n-1)/2}\dfrac{1-t}{1-t^n}\sum_{i=0}^{n-1}t^{(l+jn)(n-1-2i)/2}\\
		& = \sum_{i=0}^{n-1}t^{(n-1)/2}t^{l(n-1-2i)/2}\dfrac{1-t}{1-t^n}\sum_{j=0}^{n-1}C(n)_{n,j+1}|_{q=t^{1/2}}B(n)_{j+1,i+1}|_{q=t^{1/2}}\\
		& = \sum_{i=0}^{n-1}t^{(n-1)/2}t^{l(n-1-2i)/2} \dfrac{1-t}{1-t^n}\delta_{n,i+1}\\
		&=t^{-(l-1)(n-1)/2}\dfrac{1-t}{1-t^n},
	\end{split}
\end{equation}
 where $B(n)$ is the matrix defined in Definition \ref{def:vandermonde}, and $C(n)$ is the inverse matrix of $B(n)$.
\end{example}
 
In this sense, torus knots are the prototypical examples of the stabilization result on Alexander polynomials under adding twists. Proposition \ref{prop:stab} says the Alexander polynomials of other links behave in a similar way under adding full twists as the torus knots. 

\begin{example}
	
	\label{ex:non-zero}
	 This time we start with the unlink $L$ of $n$ components, and the diagram $\mathcal{L}$ we use for $L$ is the closure of the identity element in the braid group $B_n$ with the same orientation. Then $\mathcal{L}_m$ represents the $n$-components torus link $T_{n,mn}$. The Alexander polynomial of $T_{n,mn}$ is given by the formula: 
	 \[
	 	\Delta(T_{n,mn})=t^{-(mn-1)(n-1)/2} \dfrac{1-t}{1-t^n}(1-t^{mn})^{n-1}.
\] 
The first few terms of $\Delta(T_{n,mn})$ as $m\to \infty$ are the same as the first few terms of  \[t^{-(mn-1)(n-1)/2} \dfrac{1-t}{1-t^n}.\]
Therefore, the limiting Laurent series $h_{\mathcal{L}}(t)$ is given by \[h_{\mathcal{L}}(t) =\dfrac{1-t}{1-t^n}. \]
As in the previous example, this is the same as the $h_{\mathcal{L}}(t)$ computed using Equation \ref{eq:laurent}.
\end{example}

\bibliographystyle{amsalpha}
\bibliography{biblio}
\end{document}